\documentclass[a4paper, british]{amsart}

\usepackage{libertine}
\usepackage[libertine]{newtxmath}

\usepackage{babel}
\usepackage{enumitem}
\usepackage{hyperref}
\usepackage[utf8]{inputenc}
\usepackage{newunicodechar}
\usepackage{mathtools}
\usepackage{varioref}
\usepackage[arrow,curve,matrix]{xy}

\usepackage{colortbl}
\usepackage{graphicx}
\usepackage{tikz}

\usepackage{mathrsfs}

\definecolor{linkred}{rgb}{0.7,0.2,0.2}
\definecolor{linkblue}{rgb}{0,0.2,0.6}

\numberwithin{figure}{section}

\usepackage[hyperpageref]{backref}

\setdescription{labelindent=\parindent, leftmargin=2\parindent}
\setitemize[1]{labelindent=\parindent, leftmargin=2\parindent}
\setenumerate[1]{labelindent=0cm, leftmargin=*, widest=iiii}

\newunicodechar{א}{\ensuremath{\aleph}}
\newunicodechar{α}{\ensuremath{\alpha}}
\newunicodechar{β}{\ensuremath{\beta}}
\newunicodechar{χ}{\ensuremath{\chi}}
\newunicodechar{δ}{\ensuremath{\delta}}
\newunicodechar{ε}{\ensuremath{\varepsilon}}
\newunicodechar{Δ}{\ensuremath{\Delta}}
\newunicodechar{η}{\ensuremath{\eta}}
\newunicodechar{γ}{\ensuremath{\gamma}}
\newunicodechar{Γ}{\ensuremath{\Gamma}}
\newunicodechar{ι}{\ensuremath{\iota}}
\newunicodechar{κ}{\ensuremath{\kappa}}
\newunicodechar{λ}{\ensuremath{\lambda}}
\newunicodechar{Λ}{\ensuremath{\Lambda}}
\newunicodechar{ν}{\ensuremath{\nu}}
\newunicodechar{μ}{\ensuremath{\mu}}
\newunicodechar{ω}{\ensuremath{\omega}}
\newunicodechar{Ω}{\ensuremath{\Omega}}
\newunicodechar{π}{\ensuremath{\pi}}
\newunicodechar{Π}{\ensuremath{\Pi}}
\newunicodechar{φ}{\ensuremath{\phi}}
\newunicodechar{Φ}{\ensuremath{\Phi}}
\newunicodechar{ψ}{\ensuremath{\psi}}
\newunicodechar{Ψ}{\ensuremath{\Psi}}
\newunicodechar{ρ}{\ensuremath{\rho}}
\newunicodechar{σ}{\ensuremath{\sigma}}
\newunicodechar{Σ}{\ensuremath{\Sigma}}
\newunicodechar{τ}{\ensuremath{\tau}}
\newunicodechar{θ}{\ensuremath{\theta}}
\newunicodechar{Θ}{\ensuremath{\Theta}}
\newunicodechar{ξ}{\ensuremath{\xi}}
\newunicodechar{Ξ}{\ensuremath{\Xi}}
\newunicodechar{ζ}{\ensuremath{\zeta}}

\newunicodechar{ℓ}{\ensuremath{\ell}}
\newunicodechar{ï}{\"{\i}}

\newunicodechar{𝔸}{\ensuremath{\bA}}
\newunicodechar{𝔹}{\ensuremath{\bB}}
\newunicodechar{ℂ}{\ensuremath{\bC}}
\newunicodechar{𝔻}{\ensuremath{\bD}}
\newunicodechar{𝔼}{\ensuremath{\bE}}
\newunicodechar{𝔽}{\ensuremath{\bF}}
\newunicodechar{𝔾}{\ensuremath{\bG}}
\newunicodechar{ℕ}{\ensuremath{\bN}}
\newunicodechar{ℙ}{\ensuremath{\bP}}
\newunicodechar{ℚ}{\ensuremath{\bQ}}
\newunicodechar{ℝ}{\ensuremath{\bR}}
\newunicodechar{𝕏}{\ensuremath{\bX}}
\newunicodechar{ℤ}{\ensuremath{\bZ}}
\newunicodechar{𝒜}{\ensuremath{\sA}}
\newunicodechar{ℬ}{\ensuremath{\sB}}
\newunicodechar{𝒞}{\ensuremath{\sC}}
\newunicodechar{𝒟}{\ensuremath{\sD}}
\newunicodechar{ℰ}{\ensuremath{\sE}}
\newunicodechar{ℱ}{\ensuremath{\sF}}
\newunicodechar{𝒢}{\ensuremath{\sG}}
\newunicodechar{ℋ}{\ensuremath{\sH}}
\newunicodechar{𝒥}{\ensuremath{\sJ}}
\newunicodechar{ℒ}{\ensuremath{\sL}}
\newunicodechar{ℳ}{\ensuremath{\sM}}
\newunicodechar{𝒪}{\ensuremath{\sO}}
\newunicodechar{𝒬}{\ensuremath{\sQ}}
\newunicodechar{𝒮}{\ensuremath{\sS}}
\newunicodechar{𝒯}{\ensuremath{\sT}}
\newunicodechar{𝒲}{\ensuremath{\sW}}

\newunicodechar{∂}{\ensuremath{\partial}}
\newunicodechar{∇}{\ensuremath{\nabla}}

\newunicodechar{↺}{\ensuremath{\circlearrowleft}}
\newunicodechar{∞}{\ensuremath{\infty}}
\newunicodechar{⊕}{\ensuremath{\oplus}}
\newunicodechar{⊗}{\ensuremath{\otimes}}
\newunicodechar{•}{\ensuremath{\bullet}}
\newunicodechar{Λ}{\ensuremath{\wedge}}
\newunicodechar{↪}{\ensuremath{\into}}
\newunicodechar{→}{\ensuremath{\to}}
\newunicodechar{↦}{\ensuremath{\mapsto}}
\newunicodechar{⨯}{\ensuremath{\times}}
\newunicodechar{∪}{\ensuremath{\cup}}
\newunicodechar{∩}{\ensuremath{\cap}}
\newunicodechar{⊋}{\ensuremath{\supsetneq}}
\newunicodechar{⊇}{\ensuremath{\supseteq}}
\newunicodechar{⊃}{\ensuremath{\supset}}
\newunicodechar{⊊}{\ensuremath{\subsetneq}}
\newunicodechar{⊆}{\ensuremath{\subseteq}}
\newunicodechar{⊂}{\ensuremath{\subset}}
\newunicodechar{⊄}{\ensuremath{\not \subset}}
\newunicodechar{≥}{\ensuremath{\geq}}
\newunicodechar{≠}{\ensuremath{\neq}}
\newunicodechar{≫}{\ensuremath{\gg}}
\newunicodechar{≪}{\ensuremath{\ll}}

\newunicodechar{≤}{\ensuremath{\leq}}
\newunicodechar{∈}{\ensuremath{\in}}
\newunicodechar{∉}{\ensuremath{\not \in}}
\newunicodechar{∖}{\ensuremath{\setminus}}
\newunicodechar{◦}{\ensuremath{\circ}}
\newunicodechar{°}{\ensuremath{^\circ}}
\newunicodechar{…}{\ifmmode\mathellipsis\else\textellipsis\fi}
\newunicodechar{·}{\ensuremath{\cdot}}
\newunicodechar{⋯}{\ensuremath{\cdots}}
\newunicodechar{∅}{\ensuremath{\emptyset}}
\newunicodechar{⇒}{\ensuremath{\Rightarrow}}

\newunicodechar{⁰}{\ensuremath{^0}}
\newunicodechar{¹}{\ensuremath{^1}}
\newunicodechar{²}{\ensuremath{^2}}
\newunicodechar{³}{\ensuremath{^3}}
\newunicodechar{⁴}{\ensuremath{^4}}
\newunicodechar{⁵}{\ensuremath{^5}}
\newunicodechar{⁶}{\ensuremath{^6}}
\newunicodechar{⁷}{\ensuremath{^7}}
\newunicodechar{⁸}{\ensuremath{^8}}
\newunicodechar{⁹}{\ensuremath{^9}}
\newunicodechar{ⁱ}{\ensuremath{^i}}
\newunicodechar{⁺}{\ensuremath{^+}}

\newunicodechar{⌈}{\ensuremath{\lceil}}
\newunicodechar{⌉}{\ensuremath{\rceil}}
\newunicodechar{⌊}{\ensuremath{\lfloor}}
\newunicodechar{⌋}{\ensuremath{\rfloor}}

\newunicodechar{≅}{\ensuremath{\cong}}
\newunicodechar{⇔}{\ensuremath{\Leftrightarrow}}
\newunicodechar{∃}{\ensuremath{\exists}}
\newunicodechar{±}{\ensuremath{\pm}}

\DeclareFontFamily{OMS}{rsfs}{\skewchar\font'60}
\DeclareFontShape{OMS}{rsfs}{m}{n}{<-5>rsfs5 <5-7>rsfs7 <7->rsfs10 }{}
\DeclareSymbolFont{rsfs}{OMS}{rsfs}{m}{n}
\DeclareSymbolFontAlphabet{\scr}{rsfs}
\DeclareSymbolFontAlphabet{\scr}{rsfs}

\DeclareFontFamily{U}{mathx}{\hyphenchar\font45}
\DeclareFontShape{U}{mathx}{m}{n}{
      <5> <6> <7> <8> <9> <10>
      <10.95> <12> <14.4> <17.28> <20.74> <24.88>
      mathx10
      }{}
\DeclareSymbolFont{mathx}{U}{mathx}{m}{n}
\DeclareFontSubstitution{U}{mathx}{m}{n}
\DeclareMathAccent{\wcheck}{0}{mathx}{"71}

\DeclareMathOperator{\Aut}{Aut}

\DeclareMathOperator{\Id}{Id}

\DeclareMathOperator{\img}{img}
\DeclareMathOperator{\Pic}{Pic}

\DeclareMathOperator{\Ramification}{Ramification}

\DeclareMathOperator{\reg}{reg}

\DeclareMathOperator{\supp}{supp}

\newcommand{\sA}{\scr{A}}
\newcommand{\sB}{\scr{B}}
\newcommand{\sC}{\scr{C}}
\newcommand{\sD}{\scr{D}}
\newcommand{\sE}{\scr{E}}
\newcommand{\sF}{\scr{F}}
\newcommand{\sG}{\scr{G}}
\newcommand{\sH}{\scr{H}}

\newcommand{\sJ}{\scr{J}}
\newcommand{\sK}{\scr{K}}
\newcommand{\sL}{\scr{L}}
\newcommand{\sM}{\scr{M}}

\newcommand{\sO}{\scr{O}}
\newcommand{\sP}{\scr{P}}
\newcommand{\sQ}{\scr{Q}}

\newcommand{\sS}{\scr{S}}
\newcommand{\sT}{\scr{T}}

\newcommand{\sW}{\scr{W}}

\newcommand{\cC}{\mathcal C}

\newcommand{\bA}{\mathbb{A}}
\newcommand{\bB}{\mathbb{B}}
\newcommand{\bC}{\mathbb{C}}
\newcommand{\bD}{\mathbb{D}}
\newcommand{\bE}{\mathbb{E}}
\newcommand{\bF}{\mathbb{F}}
\newcommand{\bG}{\mathbb{G}}

\newcommand{\bN}{\mathbb{N}}

\newcommand{\bP}{\mathbb{P}}
\newcommand{\bQ}{\mathbb{Q}}
\newcommand{\bR}{\mathbb{R}}

\newcommand{\bX}{\mathbb{X}}

\newcommand{\bZ}{\mathbb{Z}}

\theoremstyle{plain}
\newtheorem{thm}{Theorem}[section]

\newtheorem{defn}[thm]{Definition}

\newtheorem{lem}[thm]{Lemma}

\newtheorem{prop}[thm]{Proposition}

\theoremstyle{remark}

\newtheorem{claim}[thm]{Claim}
\newtheorem{c-n-d}[thm]{Claim and Definition}

\newtheorem{explanation}[thm]{Explanation}
\newtheorem{notation}[thm]{Notation}

\newtheorem{rem}[thm]{Remark}

\newtheorem*{rem-nonumber}{Remark}
\newtheorem{setting}[thm]{Setting}

\numberwithin{equation}{thm}

\setlist[enumerate]{label=(\thethm.\arabic*), before={\setcounter{enumi}{\value{equation}}}, after={\setcounter{equation}{\value{enumi}}}}

\newcommand{\into}{\hookrightarrow}

\newcommand{\wtilde}{\widetilde}
\newcommand{\what}{\widehat}

\newcommand\CounterStep{\addtocounter{thm}{1}\setcounter{equation}{0}}

\newcommand{\factor}[2]{\left. \raise 2pt\hbox{$#1$} \right/\hskip -2pt\raise -2pt\hbox{$#2$}}
\newcommand{\Preprint}[1]{#1}
\newcommand{\Publication}[1]{}

\newcommand{\subversionInfo}{}
\newcommand{\svnid}[1]{}
\newcommand{\approvals}[2][Approval]{}
\renewcommand{\phi}{\varphi}
\graphicspath{{../}}

\usepackage{cancel}
\usepackage{tikz-cd}
\tikzset{commutative diagrams/arrow style=tikz}
\author{Stefan Kebekus}
\address{Stefan Kebekus, Mathematisches Institut, Albert-Ludwigs-Universität Freiburg, Ernst-Zermelo-Straße 1, 79104 Freiburg im Breisgau, Germany}
\email{\href{mailto:stefan.kebekus@math.uni-freiburg.de}{stefan.kebekus@math.uni-freiburg.de}}
\urladdr{\url{https://cplx.vm.uni-freiburg.de}}

\author{Erwan Rousseau}
\address{Erwan Rousseau, Univ Brest, CNRS UMR 6205,	Laboratoire de Mathematiques de Bretagne Atlantique\\ F-29200 Brest, France}
\email{\href{mailto:erwan.rousseau@univ-brest.fr}{erwan.rousseau@univ-brest.fr}}
\urladdr{\href{http://eroussea.perso.math.cnrs.fr}{http://eroussea.perso.math.cnrs.fr}}

\thanks{This work started during the visit of Erwan Rousseau to the Freiburg
Institute for Advanced Studies, supported by the European Unions Horizon 2020
research and innovation program under the Marie Sklodowska-Curie grant agreement
No 75434.  Rousseau thanks the Institute for providing an excellent working
environment}



\title{Entire curves in $\cC$-pairs with large irregularity}
\date{\today}

\makeatletter
\hypersetup{
  pdfauthor={\authors},
  pdftitle={\@title},
  pdfsubject={\@subjclass},
  pdfkeywords={\@keywords},
  pdfstartview={Fit},
  pdfpagelayout={TwoColumnRight},
  pdfpagemode={UseOutlines},
  bookmarks,
  colorlinks,
  linkcolor=linkblue,
  citecolor=linkred,
  urlcolor=linkred
}
\makeatother

\DeclareMathOperator{\alb}{alb}
\DeclareMathOperator{\Alb}{Alb}

\DeclareMathOperator{\diff}{d}
\DeclareMathOperator{\Div}{Div}
\DeclareMathOperator{\Fix}{Fix}

\DeclareMathOperator{\ord}{ord}

\def\tsum{\mathop{\textstyle\sum}}
\def\ae{\hspace{.5cm}\Vert}

\theoremstyle{plain}

\theoremstyle{remark}

\newtheorem{conj}[thm]{Conjecture}

\newtheorem{reminder}[thm]{Reminder}
\newtheorem{setnot}[thm]{Setting and Notation}
\newtheorem{notcho}[thm]{Notation and Choice}

\begin{document}

\approvals[Approval for Abstract]{Erwan & yes \\ Stefan & yes}
\begin{abstract}
\selectlanguage{british}

This paper extends the fundamental theorem of Bloch-Ochiai to the context of
$\cC$-pairs: If $(X, D)$ is a $\cC$-pair with large irregularity, then no entire
$\cC$-curve in $X$ is ever dense.  In its most general form, the paper's main
theorem applies to normal Kähler pairs with arbitrary singularities.  However,
it also strengthens known results for compact Kähler manifolds without boundary,
as it applies to some settings that the classic Bloch-Ochiai theorem does not
address.

The proof builds on work of Kawamata, Ueno, and Noguchi, recasting parabolic
Nevanlinna theory as a ``Nevanlinna theory for $\cC$-pairs''.  We hope the
approach might be of independent interest.

\end{abstract}

\maketitle
\tableofcontents

%
%
\svnid{$Id: 01-intro.tex 950 2024-11-11 12:47:00Z kebekus $}
\selectlanguage{british}

\section{Hyperbolicity properties of pairs with large irregularity}
\subversionInfo

\subsection{Degeneracy of entire curves}
\approvals{Erwan & yes \\ Stefan & yes}

The Albanese variety is a fundamental tool in the study of entire curves (or
rational points) in projective varieties.  Its usefulness is illustrated by the
Bloch-Ochiai Theorem.

\begin{thm}[\protect{Bloch-Ochiai Theorem, \cite[Thm.~2]{Kawa80}}]\label{thm:1-1}%
  Let $X$ be a projective manifold.  If the irregularity of $X$ is larger than
  the dimension, $q(X) > \dim X$, then every entire curve $ℂ → X$ is
  algebraically degenerate.  \qed
\end{thm}

\begin{reminder}
  An entire curve is a holomorphic morphism $ℂ → X$.  An entire curve is
  algebraically degenerate if the image of $ℂ$ is not Zariski dense in $X$.
\end{reminder}

We recall the main ideas of the proof: start with the Albanese morphism $a: X →
A$ and let $I ⊆ a(X)$ be the largest Abelian subvariety of $A$ whose action
stabilizes $\img(a)$ and consider the quotient $B := A/I$.  The image of $X$ in
$B$ is then of general type, which reduces the problem to a study of entire
curves in general type subvarieties of Abelian varieties.  With these
preparations, the Bloch-Ochiai Theorem~\ref{thm:1-1} is then an easy consequence
of the following result, which follows for instance from \cite[Thms.~4.8.2 and
2.5.4]{MR3156076}\footnote{See \cite[p.~155]{MR3156076} for further
explanation.}.

\begin{thm}[Entire curves in varieties of general type]\label{thm:1-3} %
  Let $W$ be a projective manifold of general type.  If the Albanese morphism
  $\alb(W): W → \Alb(W)$ is generically injective, then every entire curve $ℂ →
  W$ is algebraically degenerate.  \qed
\end{thm}

Theorems~\ref{thm:1-1} and \ref{thm:1-3} have both been generalized to the
setting of logarithmic pairs $(X,D)$, where $X$ is a compact Kähler manifold and
$D$ is a reduced divisor (not necessarily with simple normal crossing support).
We refer the reader to \cite[Sect.~4.8 and Thm.~4.8.17]{MR3156076} for precise
statements and for a brief history of the problem.

\subsection{Degeneracy of $\cC$-entire curves}
\approvals{Erwan & yes \\ Stefan & yes}

This paper generalizes Theorem~\ref{thm:1-1} to the setting of $\cC$-pairs
$(X,D)$, where $X$ is a normal Kähler space and $D$ is a Weil $ℚ$-divisor of the
form
\[
  D = \sum_i \frac{m_i - 1}{m_i}·D_i, \quad \text{all } m_i ∈ ℕ^{≥ 2} ∪ \{ ∞ \}.
\]
Originally introduced under the name ``geometric orbifold'' by Campana,
$\cC$-pairs interpolate between compact spaces and spaces with logarithmic
boundary.  By now, $\cC$-pairs and the derived notions such as ``adapted
differentials'' and ``$\cC$-cotangent sheaves'' are standard tools in complex,
algebraic and arithmetic geometry, with applications ranging from the geometry
of moduli spaces to the study of rational points over function fields,
\cite{MR3949026, KPS22}.  We recall the relevant notions in brief and refer the
reader to \cite{orbiAlb1, orbiAlb2} for references and a detailed introduction.

\begin{defn}[Entire $\cC$-curve, algebraic degeneracy]
  Let $(X, D)$ be a $\cC$-pair.  An \emph{entire $\cC$-curve} is a holomorphic
  morphism of $\cC$-pairs, $(ℂ, 0) → (X, D)$.  An entire $\cC$-curve is
  \emph{algebraically degenerate} if the image of $ℂ$ is not Zariski dense in
  $X$.
\end{defn}

\begin{reminder}[Morphisms of $\cC$-pairs]
  Morphisms of $\cC$-pairs are formally introduced and discussed at length in
  \cite[Sect.~7ff]{orbiAlb1}.  In the simplest case, where $X$ is smooth and $D$
  is a smooth prime divisor of the form
  \[
    D = \frac{m_1 - 1}{m_1}·D_1, \quad \text{for } m_1 ∈ ℕ^{≥ 2},
  \]
  $\cC$-pairs impose tangency conditions reminiscent of (but different from) the
  morphism $γ : ℂ → X$ is a $\cC$-morphism between the $\cC$-pairs $(ℂ, 0)$ and
  tangency conditions imposed by root stacks, \cite{MR2306040}: a holomorphic
  $(X, D)$ if
  \[
    γ^* D ≥ m_i·\supp (γ^* D).
  \]
  In other words, every intersection point of curve $γ$ and the divisor $D$ must
  have multiplicity $m_i$ at least (and not necessarily divisible by $m_i$ as in
  the case of root stacks).
\end{reminder}

\begin{reminder}[Augmented Albanese irregularity]
  Let $(X, D)$ be a $\cC$-pair where $X$ is a compact Kähler space.  Generalizing
  the \emph{augmented irregularity} of a compact manifold,
  \cite[Sect.~6]{orbiAlb2} introduces the ``augmented Albanese irregularity''
  $q⁺_{\Alb}(X, D)$, with values in $ℕ ∪ \{ ∞ \}$.
\end{reminder}

\begin{thm}[\protect{$\cC$-version of the Bloch-Ochiai theorem, see Proposition~\ref{prop:5-2}}]\label{thm:1-7}%
  Let $(X, D)$ be a $\cC$-pair where $X$ is a compact Kähler space.  If
  $q⁺_{\Alb}(X, D) > \dim X$, then every $\cC$-entire curve $(ℂ,0) → (X,D)$ is
  algebraically degenerate.
\end{thm}

\begin{rem}[Assumptions in Theorem~\ref{thm:1-7}]
  Theorem~\ref{thm:1-7} does \emph{not} assume that $X$ is smooth, but the
  definition of $\cC$-pair does require that $X$ is normal,
  \cite[Def.~2.13]{orbiAlb1}.  Theorem~\ref{thm:1-7} explicitly covers the case
  where $q⁺_{\Alb}(X, D) = ∞$.
\end{rem}

\begin{rem}[Comparison with Theorem~\ref{thm:1-1}]
  The irregularity of a compact Kähler manifold $X$ is never larger than the
  augmented Albanese irregularity of the trivial $\cC$-pair $(X,0)$,
  \begin{equation}\label{eq:1-9-1}
    q(X) ≤ q⁺_{\Alb}(X, 0).
  \end{equation}
  Even in case where $X$ is smooth and $D = 0$, Theorem~\ref{thm:1-7} is
  therefore a priori stronger than Theorem~\ref{thm:1-1} (and its logarithmic
  version \cite[Thm.~4.8.17]{MR3156076}) and can prove hyperbolicity in settings
  where Theorem~\ref{thm:1-1} does not apply.
  
  A very first example is found in \cite[Sect.~5]{Cam05}, where Campana
  constructs a simply connected surface $X$ of general type, with $q(X) = 0$ and
  $q⁺_{\Alb}(X, 0) = ∞$.  While hyperbolicity of Campana's surface can arguably
  be established through arguments more elementary than Theorem~\ref{thm:1-7},
  we strongly expect that variants of the construction will give interesting,
  higher-dimensional examples to which classic theory does not apply.  The
  reader is referred to \cite[Sect.~11]{KPS22} for a first generalization of
  Campana's idea.
\end{rem}

\begin{rem}[Theorem~\ref{thm:1-7} and the existence of an Albanese]
  Recall from \cite[Sect.~10]{orbiAlb2} that a $\cC$-pair has an Albanese with
  good universal properties if and only if its augmented Albanese irregularity
  is finite.  In this sense, Theorem~\ref{thm:1-7} can be seen as saying that
  either
  \begin{itemize}
    \item an Albanese of $\cC$-pair $(X,D)$ exists and has rather small
    dimension, or

    \item $(X,D)$ has strong hyperbolicity properties.
  \end{itemize}
\end{rem}

At its core, our proof of Theorem~\ref{thm:1-7} re-interprets parabolic
Nevanlinna theory as a ``Nevanlinna theory for $\cC$-pairs''.  We hope that the
reader might find this of independent interest.

\subsection{Perspective: Specialness and $\cC$-entire curves}
\approvals{Erwan & yes \\ Stefan & yes}

To put Theorem~\ref{thm:1-7} into perspective, we recall a famous conjecture of
Campana that relates specialness to the existence of dense entire $\cC$-curves.

\begin{conj}[\protect{Specialness and $\cC$-entire curves, \cite[Conj.~13.17]{MR2831280}}]\label{conj:1-11}%
  Let $(X, D)$ be a snc $\cC$-pair where $X$ is projective or compact Kähler.
  Then, the pair $(X, D)$ is special if and only if it admits a Zariski-dense
  entire $\cC$-curve.
\end{conj}

\begin{rem}
  Even if one is only interested in the statement for varieties (that is, the
  case where $D=0$), the use of the word ``special'' implicitly turns
  Conjecture~\ref{conj:1-11} into a statement about $\cC$-pairs.  Any result
  towards Conjecture~\ref{conj:1-11} will necessarily need to take $\cC$-pairs
  into account.
\end{rem}

If $(X, D)$ is a special Kähler $\cC$-pair, we have seen in
\cite[Rem.~8.4]{orbiAlb2} that the augmented irregularity is bounded by the
dimension, $q⁺_{\Alb}(X, D)≤ \dim X$.  In particular, Conjecture~\ref{conj:1-11}
predicts that $\cC$-pairs with $q⁺_{\Alb}(X, D) > \dim X$ have no Zariski dense
entire curves.  This is exactly the content of Theorem~\ref{thm:1-7}.  In cases
where $X$ is smooth and $D$ is empty or where $D$ is reduced, this is exactly
the logarithmic analogue of the Bloch-Ochiai Theorem~\ref{thm:1-1}.  We refer
the reader to \cite[Thm.~4.8.17]{MR3156076} for details and for a discussion.

\subsection{Acknowledgements}
\approvals{Erwan & yes\\ Stefan & yes}

We would like to thank Oliver Bräunling, Lukas Braun, Michel Brion, Johan
Commelin, Andreas Demleitner, and Wolfgang Soergel for long discussions.  Pedro
Núñez pointed us to several mistakes in early versions of the paper.  Jörg
Winkelmann patiently answered our questions throughout the work on this project.

The proof of Theorem~\ref{thm:1-7} follows ideas of Kawamata, builds on work of
Ueno and uses Nevanlinna theory, as developed by Noguchi and others.

\subsection{Global conventions}
\approvals{Erwan & yes \\ Stefan & yes}

This paper works in the category of complex analytic spaces and follows the
notation of the standard reference texts \cite{CAS, DemaillyBook, MR3156076}.  An
\emph{analytic variety} is a reduced, irreducible complex space.  For clarity,
we refer to holomorphic maps between analytic varieties as \emph{morphisms} and
reserve the word \emph{map} for meromorphic mappings.

We use the language of $\cC$-pairs, as surveyed in \cite{orbiAlb1}, and freely
refer to definitions and results from \cite{orbiAlb1} throughout the present
text.  The same holds for the paper \cite{orbiAlb2}, which introduces the core
notion of a $\cC$-semitoric pair and constructs the Albanese of a $\cC$-pair.
The reader might wish to keep hardcopies of both papers within reach.

%
%
\svnid{$Id: 02-semitori.tex 923 2024-10-02 05:05:32Z kebekus $}
\selectlanguage{british}

\section{Cyclic group actions on semitoric varieties and differentials}
\subversionInfo
\approvals{Erwan & yes \\ Stefan & yes}
\label{sec:2}

We will later need several elementary statements about actions of finite, cyclic
groups on semitoric varieties.  While certainly known to experts, we were not
able to find a suitable reference and include a full proof below.  We refer the
reader to \cite[Def.~5.3.3]{MR3156076} for the definition of ``semitoric
varieties'', and to \cite[Sect.~3]{orbiAlb2} for a detailed discussion.

\begin{setnot}\label{setnot:2-1}%
  Let $A° ⊂ A$ be a positive-dimensional semitoric variety, and let $G ⊂ \Aut(A,
  Δ_A)$ be a non-trivial, finite, cyclic group.  Then, $G$ acts on the space of
  logarithmic differentials and decomposes this space into a direct sum of
  eigenspaces.  More precisely, there exists an identification $G = ℤ/(\ord G)$
  and a unique decomposition
  \begin{equation}\label{eq:2-1-1}
    H⁰\bigl( A,\, Ω¹_A(\log Δ_A) \bigr) = \bigoplus_{0 ≤ λ < \ord G} E_{G,λ},
  \end{equation}
  where $G$ acts on every summand $E_{G,λ}$ by homotheties of the form
  \begin{equation}\label{eq:2-1-2}
    \factor{ℤ}{(\ord G)} ⨯ E_{G,λ} → E_{G,λ}, \quad \Bigl( [ℓ], τ
    \Bigr) ↦ \exp \left(ℓ·λ·\frac{2π}{\ord G}· \sqrt{-1} \right)·τ.
  \end{equation}
  Fix the identification throughout.  Recalling from \cite[Prop.~3.15]{orbiAlb2}
  that the sheaf $Ω¹_A(\log Δ_A)$ is free, the decomposition~\eqref{eq:2-1-1}
  induces a decomposition of sheaves,
  \begin{equation}\label{eq:2-1-3}
    Ω¹_A(\log Δ_A) = \bigoplus ℰ_{G,λ} %
    \quad \text{and} \quad %
    𝒯_A(-\log Δ_A) = \bigoplus ℰ^*_{G,λ}.
  \end{equation}
\end{setnot}

\begin{rem}\label{rem:2-2}%
  The summands $ℰ^*_{G,•}$ of Setting~\ref{setnot:2-1} are free.  They can
  equivalently be described as
  \[
    ℰ^*_{G,λ} = \bigcap_{μ ≠ λ} \:\: \bigcap_{τ ∈ E_{G,μ}} \ker τ.
  \]
\end{rem}

\begin{rem}
  The summands $ℰ^*_{G,•}$ of Setting~\ref{setnot:2-1} are invariant under the
  action of $A°$.  Since $A°$ is commutative as a Lie group, its Lie bracket
  vanishes and the restriction of every summand to $A°$ is a foliation.
\end{rem}

If the cyclic group $G$ of Setting~\ref{setnot:2-1} acts on $A$ by translations
with elements of $A°$, then the induced action on the space of differentials is
trivial and $H⁰\bigl( A,\, Ω¹_A(\log Δ_A) \bigr) = E_{G,0}$.  Since this is
hardly interesting, we concentrate on the case where $G$ has a fixed point and
more relevant statements can be made.  The following result is certainly not the
best possible, but suffices for our purposes.

\begin{lem}\label{lem:2-4}%
  Assume Setting~\ref{setnot:2-1}.  If the $G$-action on $A°$ has a fixed point,
  then the leaves of $ℰ^*_{G,0}|_{A°}$ are contained in the translates of
  proper, quasi-algebraic sub-semitori of $A°$.
\end{lem}
\begin{proof}
  Fix an element $0_{A°} ∈ A°$ in order to equip $A°$ with the structure of a
  Lie group.  Using that the foliation $ℰ^*_{G,0}|_{A°}$ is
  translation-invariant, it suffices to show that the leaf through $0_{A°}$ is
  contained in a proper, quasi-algebraic sub-semitorus of $A°$.  To this end,
  choose one $G$-fixed point $a ∈ A°$, choose a generator $h ∈ G$ and write
  \[
    η := L^{-1}_a◦h◦L_a ∈ \Aut(A, Δ_A),
  \]
  where $L_a ∈ \Aut(A, Δ_A)$ is the translation that sends $0_{A°}$ to $a$.
  Observe that the element $η ∈ \Aut(A, Δ_A)$ fixes $0_{A°}$.  The elements $h$
  and $η$ have the same order, and the associated cyclic groups $G = \langle
  h\rangle$ and $\langle η \rangle$ are thus canonically isomorphic.  Since
  translations act trivially on differentials, the actions of $G$ and $\langle η
  \rangle$ on $H⁰\bigl( A,\, Ω¹_A(\log Δ_A) \bigr)$ agree under this
  identification.  To prove Lemma~\ref{lem:2-4}, we can therefore replace $G$ by
  $\langle η \rangle$ and assume without loss of generality that $G ⊂ \Aut(A,
  Δ_A)$ fixes the point $0_{A°} ∈ A$ and therefore acts linearly on the tangent
  space $T_{A°}|_{\{0_{A°}\}}$.  We know what the action is: The
  Decomposition~\eqref{eq:2-1-3} induces a decomposition
  \[
    T_{A°}|_{\{0_{A°}\}} = \bigoplus T_{G,λ},
  \]
  and $G$ acts on every summand by homotheties of the form
  \[
    \factor{ℤ}{(\ord G)} ⨯ T_{G,λ} → T_{G,λ}, \quad \Bigl( [ℓ], \vec{v} \Bigr) ↦
    \exp \left(-ℓ·λ·\frac{2π}{\ord G}· \sqrt{-1} \right)·\vec{v}.
  \]
  In particular, $G$ acts trivially on the summand $T_{G,0}$.

  The exponential morphism $\exp : T_{A°}|_{\{0_{A°}\}} → A°$ of the Lie group
  $A°$ is a surjective, locally biholomorphic group morphism that is equivariant
  for the actions of $G$ on $T_{A°}|_{\{0_{A°}\}}$ and on $A°$, respectively.
  The image $\exp(T_{G,0})$ equals the leaf of $ℰ^*_{G,0}$ through $0$.  But the
  equivariant exponential morphism sends $G$-fixed points to $G$-fixed points.
  Recalling from \cite[Prop.~3.18]{orbiAlb2} that $h|_{A°} ∈ \Aut(A°)$ is a
  group morphism, this means that the leaf of $ℰ^*_{G,0}$ through $0_{A°}$ is
  then necessarily contained in
  \[
    \Fix(G) ∩ A° = \ker\bigl(h|_{A°} - \Id_{A°}\bigr) ⊆ A°.
  \]
  Recall from \cite[Facts~3.23 and 3.27]{orbiAlb2} that this is indeed a
  quasi-algebraic, proper sub-semitorus of $A°$.
\end{proof}

%
%
\svnid{$Id: 03-nevanlinna.tex 919 2024-09-30 14:16:54Z kebekus $}
\selectlanguage{british}

\section{Nevanlinna theory for branched covers of \texorpdfstring{$ℂ$}{ℂ}}
\subversionInfo
\approvals{Erwan & yes \\ Stefan & yes}

To prepare for the proof of Theorem~\ref{thm:1-7}, this section recalls a number
of useful results from Nevanlinna theory.  We refer the reader to \cite[Sect.~3
and p.~250f]{MR3331401} and \cite[Sect.~2.7]{MR3156076} for details and for a
well-written introduction to Nevanlinna theory for branched covers of $ℂ$.  To
begin, we fix setting and notation for the remainder of the present section.

\begin{setting}[Holomorphic cover of the complex plane]\label{set:3-1} %
  Let $V$ be a connected Riemann surface and $ρ : V \twoheadrightarrow ℂ$ be a
  cover (recall the convention \cite[Def.~2.21]{orbiAlb1} that covers are
  finite).  We denote the standard coordinate function on the complex line by $t
  ∈ H⁰ \bigl( ℂ,\, 𝒪_ℂ \bigr)$.  Given any real number $r ≥ 0$, let $Δ_r ⊂ ℂ$
  be the disk of radius $r$ and write $V_r := ρ^{-1}(Δ_r)$ for its preimage.
\end{setting}

\subsection{The Nevanlinna functions}
\approvals{Erwan & yes \\ Stefan & yes}

Maintain Setting~\ref{set:3-1}.  Aiming to generalize Bloch-Ochiai's
Theorem~\ref{thm:1-1}, we are interested in a criterion to guarantee that
holomorphic morphisms from $V$ to a projective manifold $Y$ are algebraically
degenerate.  The criterion, Theorem~\vref{thm:4-1}, builds on work of Noguchi
and makes heavy use the ``main Nevanlinna functions for the branched covering
$ρ$''.  We recall the definitions of the Nevanlinna functions and briefly state
their main properties and refer to \cite[Sect.~2.7]{MR3156076} for a more
detailed introduction.

\begin{reminder}[Counting functions]\label{remi:3-2} %
  In Setting~\ref{set:3-1}, let $H ∈ \Div(V)$ be any effective divisor.  We can
  then consider the following functions.
  \begin{align*}
    N_{H} : [1, ∞) & → ℝ^{≥ 0}, & r & ↦ \frac{1}{\degρ} \int_{1}^{r} \left( \tsum_{u ∈ V_s}\ord_{u} H \right) \frac{ds}{s} & & \text{Counting} \\
    N_{1,H} : [1, ∞) & → ℝ^{≥ 0}, & r & ↦ \frac{1}{\degρ} \int_{1}^{r} \left( \tsum_{u ∈ V_s}\min\{1,\ord_{u} H \} \right) \frac{ds}{s} & & \text{Truncated counting}
  \end{align*}
\end{reminder}

\begin{reminder}[Proximity and height functions]\label{remi:3-3} %
  In Setting~\ref{set:3-1}, let $g : V → Y$ be any non-constant morphism from
  $V$ to a projective manifold $Y$, equipped with a Hermitian line bundle $L :=
  \bigl( ℒ,\, |·|\bigr)$ and a section $σ ∈ H⁰\bigl( Y, ℒ \bigr)$ such that $σ◦
  g$ is not identically zero.  Writing $c_1(L)$ for the Chern form of the
  Hermitian bundle $L$, we consider the following functions,
  \begin{align*}
    m(•, g, L, σ) : [1, ∞) & → ℝ, & r & ↦ \frac{1}{\deg ρ} \int_{∂ V_r} \log\frac{1}{|σ◦g|} · ρ^* (d^c \log |t|²) & & \text{Proximity} \\
    T(•, g,L) : [1, ∞) & → ℝ, & r & ↦ \frac{1}{\deg ρ} \int_1^r \left( \int_{V_s} g^* c_1(L) \right) \frac{ds}{s} & & \text{Height}
  \end{align*}
\end{reminder}

\begin{rem}[Integral in the proximity function]
  The existence of the integral in the definition of $m(•, g, L, σ)$ is
  elementary, cf.~\cite[(2.3.30) and Sect.~2.7]{MR3156076}.  For the reader's
  convenience, we remark that our main reference, \cite{MR3156076}, writes $γ :=
  d^c \log |t|²$.  Our second main reference, \cite{MR780664}, uses the notation
  $η := ρ^* (d^c \log |t|²)$.  We will constantly use the fact that
  \begin{equation}\label{eq:3-4-1}
    \int_{∂ Δ_r} d^c \log |t|² = 1
    \quad\text{and hence}\quad
    \int_{∂ V_r} ρ^* (d^c \log |t|²) = \deg ρ.
  \end{equation}
\end{rem}

The following elementary facts are well-known to experts, cf.~\cite[p.~234 and
250]{MR3331401}.  \Preprint{We include full proofs for the reader's
convenience.}\Publication{The preprint version of this paper includes full
proofs for the reader's convenience.}

\begin{lem}[Boundedness of the proximity function]\label{lem:3-5} %
  The function $m(•, g, L, σ)$ of Reminder~\ref{remi:3-3} is bounded from below.
\end{lem}
\Preprint{
\begin{proof}
  This follows from Equation~\eqref{eq:3-4-1}, given that the continuous
  function $Y → ℝ^{≥ 0}$, $y → |σ(y)|$ on the compact space $Y$ is bounded from
  above.
\end{proof}
}

\begin{lem}[\protect{Independence on choice of metric}]\label{lem:3-6} %
  If the bundle $ℒ$ of Reminder~\ref{remi:3-3} carries two Hermitian metrics,
  $L_1 := \bigl( ℒ,\, |·|_1\bigr)$ and $L_2 := \bigl( ℒ,\, |·|_2\bigr)$, then
  \begin{align}
    \label{eq:3-6-1} m(•, g, L_1, σ) &= m(•, g, L_2, σ) + O(1) \\
    \label{eq:3-6-2} T(•, g, L_1) &= T(•, g, L_2) + O(1).
  \end{align}
\end{lem}
\Preprint{
\begin{proof}
  Equation~\eqref{eq:3-6-1} follows from \eqref{eq:3-4-1}, given that the two
  norm functions $|·|_1$ and $|·|_2 ∈ \cC⁰(ℒ)$ differ only by multiplication
  with the pull-back of a strictly positive function in $\cC⁰(V)$, which attains
  its minimum and maximum.

  The proof of \eqref{eq:3-6-2} is almost identical to \cite[proof of
  Lem.~3.1]{MR3331401}.  To begin, observe that $c_1(L_1)$ and $c_1(L_2)$ are
  smooth closed $(1,1)$-forms on $Y$ with identical cohomology class.  The
  difference $c_1(L_1) - c_1(L_2)$ is thus exact, and the $∂ \overline{∂}$-lemma
  yields a smooth function $φ$ on $V$ such that $c_1(L_1) - c_1(L_2) = d d^c φ$.
  We find
  \begin{align*}
    & T(r, g, L_1) - T(r, g, L_2) \\
    & \qquad = \frac{1}{\deg ρ} \int_1^r \left( \int_{V_s} g^* \bigl( c_1(L_1) - c_1(L_2)\bigr) \right) \frac{ds}{s} \\
    & \qquad = \frac{1}{\deg ρ} \int_1^r \left( \int_{V_s} dd^c (φ ◦ g) \right) \frac{ds}{s} \\
    & \qquad = \frac{1}{\deg ρ} \int_1^r \left( \int_{∂ V_s} d^c (φ ◦ g) \right) \frac{ds}{s} && \text{Stokes}\\
    & \qquad = \frac{1}{\deg ρ} \int_{V_r ∖ V_1 } d^c (φ ◦ g) Λ \frac{d|t◦ρ|}{|t◦ρ|} && \text{Fubini} \\
    & \qquad = \frac{1}{2·\deg ρ} \int_{V_r ∖ V_1 } d^c (φ ◦ g) Λ ρ^* (d \log |t|²) \\
    & \qquad = \frac{-1}{2·\deg ρ} \int_{V_r ∖ V_1 } d (φ ◦ g) Λ ρ^* (d^c \log |t|²) && d^c u Λ dv = d^c v Λ du \\
    & \qquad = \frac{-1}{2·\deg ρ} \int_{V_r ∖ V_1 } d \Bigl((φ ◦ g) · ρ^* (d^c \log |t|²)\Bigr) && dd^c \log |t|² = 0 \\
    & \qquad = \frac{-1}{2·\deg ρ} \Biggl(\int_{∂ V_r} (φ ◦ g)·ρ^*(d^c \log |t|²) \\
    & \qquad \qquad \qquad \qquad \qquad - \int_{∂ V_1} (φ ◦ g)·ρ^*(d^c \log |t|²) \Biggr) && \text{Stokes.}
  \end{align*}
  Since $φ$ is bounded as a continuous function on the compact manifold $Y$,
  Equation~\eqref{eq:3-4-1} implies that the integrals in the last line are
  bounded.
\end{proof}
}

\begin{lem}[Height function for ample divisor]\label{lem:3-7} %
  If the bundle $ℒ$ of Reminder~\ref{remi:3-3} is ample, then the height
  function tends to infinity.  More precisely, there exist $c⁺_1 ∈ ℝ⁺$ and $c_2
  ∈ ℝ$ such that
  \[
    T(r, g, L) ≥ c⁺_1·\log r + c_2, \quad\text{for every } r ≥ 1.
  \]
\end{lem}
\Preprint{
\begin{proof}
  Ampleness of $ℒ$ and Lemma~\ref{lem:3-6} allows replacing $|·|$ with a metric
  of positive Chern form.  The proof of \cite[p.~234]{MR3331401} will then apply
  verbatim:
  \[
    T(r, g, L) = \frac{1}{\deg ρ} \int_1^r \left( \int_{V_s} g^* c_1(L) \right) \frac{ds}{s} ≥ \frac{1}{\deg ρ} \int_1^r \left( \int_{V_1} g^* c_1(L) \right) \frac{ds}{s}
    = \operatorname{const}⁺·\log r \qedhere
  \]
\end{proof}
}

We also recall that the Nevanlinna functions of Reminders~\ref{remi:3-2} and
\ref{remi:3-3} are related to one another by the following fundamental result.

\begin{thm}[\protect{First main theorem, cf.~\cite[Thm.~2.7.4]{MR3156076}}]\label{thm:3-8} %
  In the setting of Reminder~\ref{remi:3-3}, let $D ∈ \Div(Y)$ be the
  zero-divisor of the section $σ$.  Then,
  \[
    T(•, g, L) = N_{g^*D}(•) + m(•, g, L, σ) + O(1).  \eqno \qed
  \]
\end{thm}

\subsection{The Lemma on logarithmic derivatives}
\approvals{Erwan & yes \\ Stefan & yes}

The next section develops a degeneracy criterion, Theorem~\ref{thm:4-1}, whose
proof uses a fundamental fact of Nevanlinna theory for branched covers of $ℂ$:
the ``Lemma on logarithmic derivatives''.  For the reader's convenience, we
briefly recall the statement.  The following notation will be used to compare
differentials on $V$ with the standard differential $\diff t$ on the complex
plane.

\begin{notation}[Differentials on $V$ and the standard differential on the complex line]\label{not:3-9} %
  In Setting~\ref{set:3-1}, observe that every meromorphic differential $τ ∈ H⁰
  \bigl( V,\, Ω¹_V ⊗ \sK_V \bigr)$ can be written as $ξ·(ρ^* d t)$, where $ξ ∈
  H⁰ \bigl( V, \sK_V \bigr)$ is meromorphic.  Writing $ξ =: \frac{τ}{ρ^* d t}$
  for ease of notation, we can thus define a morphism that takes meromorphic
  differentials to meromorphic functions,
  \[
    η : H⁰\bigl( V,\, Ω¹_V ⊗ \sK_V \bigr) → H⁰\bigl( V, \sK_V\bigr), \quad τ ↦
    \frac{τ}{ρ^* d t}.
  \]
\end{notation}

The Lemma on logarithmic derivatives views the meromorphic functions $ξ$ as
morphisms $ξ : V → ℙ¹$ and considers the proximity function with respect to the
standard Hermitian structure on the hyperplane bundle of $ℙ¹$.  The following
notation will be used.

\begin{notation}[Hermitian structure on the anti-tautological bundle]\label{not:3-10} %
  Denote the standard Hermitian structure on the hyperplane bundle of $ℙ¹$ by $H
  := (𝒪_{ℙ¹}(1), |·|)$.  Writing $z$ for the standard coordinate on $ℂ ⊂ ℙ¹$,
  we also consider the standard sections $σ_0, σ_{∞} ∈ H⁰\bigl( ℙ¹,\,
  𝒪_{ℙ¹}(1)\bigr)$, where $\operatorname{div}σ_{•} = •$, where $σ_0 = z·σ_{∞}$
  and
  \begin{equation}\label{eq:3-10-1}
    |σ_0(z)|² = \frac{|z|²}{|z|²+1}
    \quad\text{and}\quad
    |σ_{∞}(z)|² = \frac{1}{|z|²+1}.
  \end{equation}
\end{notation}

\begin{thm}[\protect{Lemma on logarithmic derivatives, \cite[Lem.~1.6]{MR780664}}]\label{thm:3-11} %
  In the setting of Reminder~\ref{remi:3-3}, assume that $ℒ$ is ample.  Given a
  reduced divisor $D ∈ \Div(Y)$ with $\img g ⊄ \supp D$ and a logarithmic
  differential $ω ∈ H⁰\bigl( Y,\, Ω¹_Y(\log D) \bigr)$, consider the meromorphic
  function $ξ := η\bigl( g^*ω\bigr)$, and view it as a morphism $ξ : V → ℙ¹$.
  If $ε > 0$ is any number, there exists an inequality of the following form,
  \begin{equation}\label{eq:3-11-1}
    m(•, ξ, H, σ_∞) ≤ ε·T(•, g, L) \ae.
  \end{equation}
\end{thm}

\begin{reminder}[Notation used in \eqref{eq:3-11-1}]
  As usual in Nevanlinna theory, the symbol $\Vert$ in \eqref{eq:3-11-1} means
  that the inequality holds outside a subset of $[1, ∞)$ that is a union of
  (possibly infinitely many) intervals with finite total measure.  The subset
  may well depend on the number $ε$.
\end{reminder}

\begin{proof}[Proof of Theorem~\ref{thm:3-11}]
  \CounterStep Theorem~\ref{thm:3-11} is a reformulation of
  \cite[Lem.~1.6]{MR780664}.  To begin, observe that it follows from
  Lemma~\ref{lem:3-7} that the validity of Inequality~\eqref{eq:3-11-1} depends
  only on the classes of the functions $m(•, ξ, H, σ_∞)$ and $T(•, g, L)$,
  modulo addition of bounded functions.  We use this freedom in two ways.
  \begin{itemize}
    \item Using Lemma~\ref{lem:3-6}, we may replace the Hermitian metric on the
    ample bundle $ℒ$ and assume without loss of generality that $c_1(L)$ is a
    positive form on $V$.  This will later become relevant when we invoke
    \cite[Lem.~1.6]{MR780664}, where positivity of $c_1(L)$ is an implicit
    assumption\footnote{The sentence ``we assume that $Ω$ is the positive form
    associated with a Hermitian metric $h$ on $X$'' in \cite[p.~299]{MR780664}
    contains a misprint.  The symbol ``$X$'' should read ``$V$''.}.

    \item We may replace the proximity function $m(•, ξ, H, σ_∞)$ in
    \eqref{eq:3-11-1} with the simpler variant $m(•, ξ)$ used in Noguchi's
    paper.
  \end{itemize}
  We explain the second bullet item in detail and consider the estimates
  \begin{align}
    m(r, ξ, H, σ_∞) & = \frac{1}{\deg ρ} \int_{∂ V_r} \log\frac{1}{|σ_{∞}◦ξ|} · ρ^* (d^c \log |t|²) && \text{definition}\\
     & = \frac{1}{\deg ρ} \int_{∂ V_r} \log \sqrt{|ξ|²+1} · ρ^* (d^c \log |t|²) && \text{\eqref{eq:3-10-1}} \\
    \label{eq:3-13-3} & = \underbrace{\frac{1}{\deg ρ} \int_{∂ V_r} \log⁺ |ξ| · ρ^* (d^c \log |t|²)}_{=: m(r, ξ)\text{, as defined in \cite[p.~298]{MR780664}}} + O(1), && \text{see below}
  \end{align}
  where
  \[
    \log⁺ : ℝ → ℝ^{≥ 0}, \quad r ↦
    \begin{cases}
      0 & \text{if } r < 1 \\
      \log r & \text{otherwise}.
    \end{cases}
  \]
  The estimate \eqref{eq:3-13-3} follows from \eqref{eq:3-4-1} and from the
  elementary inequality
  \[
    0 ≤ \log \sqrt{r²+1} - \log⁺ r ≤ \log \sqrt{2}, \quad \text{for every } r ∈ ℝ^{≥ 0}.
  \]
  Wrapping up what we have shown so far: To prove Theorem~\ref{thm:3-11}, it
  suffices to show that for every $ε > 0$, there exists an inequality of the
  form
  \begin{equation}\label{eq:3-13-4}
    m(•, ξ) ≤ ε·T(•, g, L) \ae.
  \end{equation}
  choose $δ ∈ (0,1)$ such that $δ ≤ ε·c⁺_1$ and recall from \cite[Lem.~1.6 and
  Given one $ε$, consider the constants $c⁺_1$ and $c_2$ of Lemma~\ref{lem:3-7},
  proof on p.~302]{MR780664} that there exists a constant $c ∈ ℝ$ and an
  inequality of the form
  \begin{equation}\label{eq:3-13-5}
    m(• ,ξ) ≤ δ·\log • + 20·\log⁺ T(•, g, L) + c \ae.
  \end{equation}
  But given that $T(•, g, L)$ is monotonous and unbounded, the following will
  hold for all sufficiently large numbers $r ≫ 0$,
  \begin{align}
    \label{eq:3-13-6} 0 ≤ 20·\log⁺ T(r, g, L) & ≤ \frac{ε}{3}·T(r, g, L), \\
    \label{eq:3-13-7} c - \frac{δ·c_2}{c⁺_1} & ≤ \frac{ε}{3}·T(r, g, L).
  \end{align}
  For these numbers sufficiently large numbers $r$, the right side of
  \eqref{eq:3-13-5} then reads
  \begin{align*}
    & δ·\log r + 20·\log⁺ T(r, g, L) + c \\
    & \qquad = \frac{δ}{c⁺_1}(c⁺_1·\log r) + 20·\log⁺ T(r, g, L) + c \\
    & \qquad ≤ \frac{δ}{c⁺_1}·T(r, g, L) + 20·\log⁺ T(r, g, L) + c - \frac{δ·c_2}{c⁺_1} && \text{Lem.~\ref{lem:3-7}} \\
    & \qquad ≤ \frac{ε}{3}·T(r, g, L) + 20·\log⁺ T(r, g, L) + c - \frac{δ·c_2}{c⁺_1} && \text{choice of $δ$ and \eqref{eq:3-13-6}} \\
    & \qquad ≤ ε·T(r, g, L) && \text{\eqref{eq:3-13-6} and \eqref{eq:3-13-7}.}
  \end{align*}
  This establishes an inequality of the desired form \eqref{eq:3-13-4} and
  finishes the proof of Theorem~\ref{thm:3-11}.
\end{proof}

%
%
\svnid{$Id: 04-degeneracy.tex 913 2024-09-30 08:35:24Z kebekus $}
\selectlanguage{british}

\section{A degeneracy criterion for entire curves}
\subversionInfo
\approvals{Erwan & yes \\ Stefan & yes}

Building on work of Noguchi, this section establishes a criterion to guarantee
algebraic degeneracy of the morphism $g$ from Reminder~\ref{remi:3-3}.

\begin{thm}[Degeneracy criterion]\label{thm:4-1} %
  In the setting of Reminder~\ref{remi:3-3}, let $D ∈ \Div(Y)$ be a reduced
  divisor with snc support, such that the following holds.
  \begin{enumerate}
    \item\label{il:4-1-1} The Albanese morphism $\alb(Y,D)°$ of the log pair is
    generically finite.

    \item\label{il:4-1-2} The image of $\alb(Y,D)°$ is a variety of log-general
    type.

    \item\label{il:4-1-3} The image of $g$ does not intersect $D$.
  \end{enumerate}
  Suppose that there exists a reduced divisor\footnote{Note: We do \emph{not}
  assume that $D_1$ has snc support.} $D_1∈ \Div(Y)$ with $\img g ⊄ \supp D_1$
  and logarithmic differentials $ω_1, …, ω_l ∈ H⁰\bigl( Y,\, Ω¹_Y(\log D_1)
  \bigr)$ such that the associated meromorphic functions $ξ_i := η\bigl( g^*ω_i
  \bigr)$ are holomorphic and do not vanish identically.  If
  \begin{equation}\label{il:4-1-4}
    \supp \bigl(\Ramification ρ \bigr) ⊆ \bigcup_{i ∈ \{1,…,l\}} \{ v ∈ V \::\: ξ_i(v) = 0 \},
  \end{equation}
  then $g$ is algebraically degenerate.
\end{thm}

\begin{explanation}
  Condition~\ref{il:4-1-2} might require a word of explanation.  To formulate
  the condition precisely, choose one Albanese and consider the map
  \[
    \alb(Y,D)°: Y° → \Alb(Y,D)° ⊂ \Alb(Y,D).
  \]
  Consider the toplogical closure $W := \overline{\img \alb(Y,D)°}$.  Observe
  that $W$ is analytic because $\alb(Y,D)°$ is quasi-algebraic, and write $W° :=
  W ∩ \Alb(Y,D)°$.  We obtain a tuple $(W, Δ)$ where $W$ is a (potentially
  non-normal) variety and $Δ = W ∖ W°$ is an analytic subset of pure codimension
  one.  Condition~\ref{il:4-1-2} says that one (equivalently: every)
  log-resolution of $(W,Δ)$ is of log-general type.  Since $\Alb(Y,D)$ is unique
  up to bimeromorphic equivalence, this condition does not depend on the choice
  made in the construction.
\end{explanation}

We prove Theorem~\ref{thm:4-1} in Section~\ref{sec:4-2} below.

\subsection{Noguchi's criterion}
\approvals{Erwan & yes \\ Stefan & yes}

The proof of Theorem~\ref{thm:4-1} relies on the following proposition.
Essentially due to Noguchi, it replaces Condition~\eqref{il:4-1-4} by an
inequality between Nevanlinna functions.  The interested reader might also want
to look at a related criterion of Yamanoi, \cite[Prop.~3.3]{MR2667786}, that is
stronger but works only in the compact case.

\begin{prop}[Noguchi's criterion]\label{prop:4-3} %
  In the setting of Reminder~\ref{remi:3-3}, let $D ∈ \Div(Y)$ be a reduced
  divisor with snc support, such that the following holds.
  \begin{enumerate}
    \item The Albanese morphism $\alb(Y,D)°$ of the log pair $(Y,D)$ is
    generically finite.

    \item\label{il:4-3-2} The image of $\alb(Y,D)°$ is a variety of log-general
    type.

    \item The image of $g$ does not intersect $D$.
  \end{enumerate}
  If the line bundle $ℒ ∈ \Pic(Y)$ is ample and if the inequality
  \begin{equation}\label{eq:4-3-4}
    N_{\Ramification ρ}(•) ≤ ε·T(•, g, L) \ae
  \end{equation}
  holds for every $ε >0$, then $g$ is algebraically degenerate.
\end{prop}
\begin{proof}
  \CounterStep We argue by contradiction and assume that the image of $g$
  \emph{is} Zariski dense in $Y$.  By \cite[Thm.~3.2 on p.~306]{MR780664}, there
  will then exist constants $c⁺_1, c⁺_2, c⁺_3 ∈ ℝ⁺$ and $c_4 ∈ ℝ$ such that an
  inequality of the form
  \begin{equation}\label{eq:4-4-1}
    c⁺_1·T(•, g, L) ≤ N_{\Ramification ρ}(•) + c⁺_2·ε·\log • + c⁺_3·\log⁺ T(•, g, L) + c_4 \ae
  \end{equation}
  holds for every number $ε ∈ (0,1)$.  Choose $ε$ small enough so that
  \[
    (1 + c⁺_2 + c⁺_3)·ε < c⁺_1
  \]
  and use the assumption that $ℒ$ is ample to observe
  \begin{align*}
    c⁺_1·T(•, g, L) & ≤ N_{\Ramification ρ}(•) + c⁺_2·ε·\log • + c⁺_3·\log⁺ T(•, g, L) + c_4 \ae && \text{\eqref{eq:4-4-1}} \\
    & ≤ ε·T(•, g, L) + c⁺_2·ε·\log • + c⁺_3·\log⁺ T(•, g, L) + c_4 \ae && \text{\eqref{eq:4-3-4}} \\
    & ≤ ε·T(•, g, L) + c⁺_2·ε·T(•, g, L) + c⁺_3·\log⁺ T(•, g, L) + c_4 \ae && \text{Lem.~\ref{lem:3-7}} \\
    & ≤ ε·T(•, g, L) + c⁺_2·ε·T(•, g, L) + c⁺_3·ε·T(•, g, L) + c_4 \ae && \text{Lem.~\ref{lem:3-7}} \\
    & = (1 + c⁺_2 + c⁺_3)·ε·T(•, g, L).
  \end{align*}
  Given that $T(•, g, L)$ is monotonous and unbounded, this is absurd.
\end{proof}

\subsection{Proof of Theorem~\ref*{thm:4-1}}
\approvals{Erwan & yes \\ Stefan & yes}
\label{sec:4-2}

Since none of our assumptions refers to $L$, we may assume without loss of
generality that $ℒ$ is ample.  Following \cite[proof of Prop.~3.1]{MR2667786},
we aim to apply Theorem~\ref{thm:3-11} (``Lemma on logarithmic derivatives'').
To this end, consider the standard Hermitian bundle $H$ of
Notation~\ref{not:3-10}.
  
Using Assumption~\eqref{il:4-1-4}, the counting function for the ramification of
$ρ$ is estimated as follows,
\begin{align*}
  N_{\Ramification ρ}(•) & ≤ (\deg ρ)·N_{1,\Ramification ρ}(•) & & \forall v ∈ V: \ord_v \operatorname{Ram.} ρ ≤ \deg ρ \\
  & ≤ (\deg ρ)·\sum_{i=1}^l N_{ξ_i^*\{0\}}(•) & & \text{Ass.~\eqref{il:4-1-4}}
  \intertext{Given any $ε' > 0$, we can give an estimate for each summand,}
  N_{ξ_i^*\{0\}}(•) & = T(•, ξ_i, H) - m(•, ξ_i, H, σ_0) +O(1) & & \text{Thm.~\ref{thm:3-8} (``first main'')} \\
                    & ≤ T(•, ξ_i, H) + O(1) & & \text{Lem.~\ref{lem:3-5}} \\
                    & = N_{ξ_i^*\{∞\}}(•) + m(•, ξ_i, H, σ_{∞}) + O(1) & & \text{Thm.~\ref{thm:3-8} (``first main'')} \\
                    & = m(•, ξ_i, H, σ_{∞}) + O(1) & & \text{since $ξ_i$ is holomorphic} \\
                    & ≤ ε'·T(•, g, L) + O(1) \ae & & \text{Thm.~\ref{thm:3-11} (``log.~derivatives'')}
\end{align*}
Lemma~\ref{lem:3-7} will then imply that Inequality~\eqref{eq:4-3-4} of
Noguchi's criterion holds for all $ε >0$.  The claim thus follows.  \qed

%
%
\svnid{$Id: 05-BO.tex 943 2024-10-14 10:58:51Z kebekus $}
\selectlanguage{british}

\section{\texorpdfstring{$\cC$}{𝒞}-version of the Bloch-Ochiai theorem, proof of Theorem~\ref*{thm:1-7}}
\subversionInfo
\approvals{Erwan & yes \\ Stefan & yes}

Theorem~\ref{thm:1-7} is a direct consequence of the following, stronger
Proposition~\ref{prop:5-2}.  The formulation of Proposition~\ref{prop:5-2} use
the ``Albanese of a cover'', as introduced and discussed in
\cite[Sect.~5]{orbiAlb2}.  For the reader's convenience, we recall the relevant
notions in brief.

\begin{reminder}[\protect{The Albanese of a cover, \cite[Def.~5.2 and Prop.~5.4]{orbiAlb2}}]%
  Let $(X, D)$ be a $\cC$-pair where $X$ is compact Kähler and let $γ : \what{X}
  \twoheadrightarrow X$ be a cover.  Consider the open sets
  \[
    X° := X ∖ \supp ⌊ D ⌋
    \quad\text{and}\quad
    \what{X}° := γ^{-1}(X°).
  \]
  Then, there exists a semitoric variety $\Alb(X,D,γ)° ⊂ \Alb(X,D,γ)$ and a
  quasi-algebraic\footnote{Quasi-algebraic = the holomorphic morphism
  $\alb(X,D,γ)°$ extends to a meromorphic map between the compact spaces
  $\what{X}$ and $\Alb(X,D,γ)$} morphism
  \[
    \alb(X,D,γ)° : \what{X}° → \Alb(X,D,γ)°
  \]
  such that the pull-back morphism for logarithmic reflexive differentials takes
  its image in the subspace of adapted reflexive differentials.  Moreover, if
  $A° ⊂ A$ is any semitoric variety, if $a° : \what{X}° → A°$ is any
  quasi-algebraic morphism whose associated pull-back morphism for logarithmic
  reflexive differentials takes its image in the subspace of adapted reflexive
  differentials, then $a$ factors uniquely as
  \[
    \begin{tikzcd}[column sep=2cm]
      \what{X}° \arrow[r, "\alb(X{,}D{,}γ)°"'] \arrow[rr, "a°", bend left=15] & \Alb(X,D,γ)° \arrow[r, "∃!b°"'] & A°,
    \end{tikzcd}
  \]
  where $b°$ is quasi-algebraic.
\end{reminder}

\begin{prop}[Strong version of Theorem~\ref{thm:1-7}]\label{prop:5-2}%
  Let $(X, D)$ be a $\cC$-pair where $X$ is compact Kähler.  If there exists a
  cover $γ : \what{X} \twoheadrightarrow X$ such that $\alb(X,D,γ)°$ is
  \emph{not} dominant, then every $\cC$-entire curve $(ℂ,0) → (X,D)$ is
  algebraically degenerate.
\end{prop}

\begin{rem}[Quasi-algebraicity and dominance]
  The morphism $\alb(X,D,γ)°$ is quasi-algebraic and its image is therefore
  constructible.  The word ``dominant'' in Proposition~\ref{prop:5-2} is
  therefore meaningful.
\end{rem}

\begin{rem}[Relation to Theorem~\ref{thm:1-7}]
  In the setting of Theorem~\ref{thm:1-7}, the assumption $q⁺_{\Alb}(X, D) >
  \dim X$ guarantees the existence of a cover $γ : \what{X} \twoheadrightarrow
  X$ such that $\dim \Alb(X,D,γ) > \dim X$.  In particular, $\alb(X,D,γ)°$ is
  \emph{not} dominant in this setting.
\end{rem}

We will prove Proposition~\ref{prop:5-2} in the remainder of the present
section.

\subsection{Proof of Proposition~\ref*{prop:5-2}}
\approvals{Erwan & yes \\ Stefan & yes}
\CounterStep

We argue by contradiction and assume that there exists one $\cC$-entire curve
$φ: (ℂ,0) → (X,D)$ whose image is Zariski dense in $X$.  As before, consider the
open sets
\[
  X° := X ∖ \supp ⌊ D ⌋
  \quad\text{and}\quad
  \what{X}° := γ^{-1}(X°).
\]
The definition of $\cC$-morphism guarantees that $φ$ takes its image in $X° ⊆
X$.

\subsection*{Step 1: Galois closure and the Albanese}
\approvals{Erwan & yes \\ Stefan & yes}

Functoriality of the Albanese, as spelled out in \cite[Lem.~5.5]{orbiAlb2},
allows replacing the cover $γ$ by its Galois closure.  We will therefore assume
without loss of generality that $γ$ is Galois, with group $G$.

The proof of Proposition~\ref{prop:5-2} discusses the Albanese of the cover $γ$.
For brevity of notation, we denote the associated semitoric variety by $\Alb° ⊂
\Alb$ and write $\alb° : \what{X}° → \Alb°$ for the associated morphism.  Recall
from \cite[Obs.~5.6]{orbiAlb2} that $G$ acts on $\Alb°$ by quasi-algebraic
automorphisms, in a way that makes the morphism $\alb°$ equivariant.  Finally,
choose an element $\what{x} ∈ \what{X}°$ and use its image point
\[
  0_{\Alb°} := \alb°(\what{x}) ∈ \Alb°
\]
to equip $\Alb°$ with the structure of a Lie group.

\subsection*{Step 2: Reminder}
\approvals{Erwan & yes \\ Stefan & yes}

The assumption that $\alb°$ is \emph{not} dominant allows using constructions
and results of our earlier paper \cite{orbiAlb2}.  For the reader's convenience,
we recall the main points.

\subsubsection*{Construction of a semitoric quotient variety}

In \cite[Const.~8.10]{orbiAlb2}, we construct a non-trivial semitoric variety
$B° ⊆ B$ with $G$-action, a point $0_{B°} ∈ B°$ that equips $B°$ with the
structure of a Lie group, and a diagram
\begin{equation}\label{eq:5-1}
  \begin{tikzcd}[column sep=2cm, row sep=1cm]
    \what{X}° \ar[r, "\alb°"'] \ar[d, two heads, "γ°\text{, quotient by }G"'] \ar[rr, bend left=15, "b°"] & \Alb° \ar[r, two heads, "β°\text{, group quotient}"'] \ar[d, two heads, "γ_{\Alb°}\text{, quotient by }G"] & B° \ar[d, two heads, "γ_{B°}\text{, quotient by }G"] \\
    X° \ar[r, "δ°"'] & \factor{\Alb°}{G} \ar[r, two heads, "ε°"'] & \factor{B°}{G}
  \end{tikzcd}
\end{equation}
where (among other things) the following holds.
\begin{enumerate}
  \item All horizontal arrows are quasi-algebraic,

  \item all arrows in the top row are $G$-equivariant, and

  \item all arrows in the bottom row are $\cC$-morphisms for the $\cC$-pairs
    \[
      (X°, D°),\quad \factor{\bigl( \Alb°, 0\bigr)}{G}, \quad\text{and}\quad \factor{\bigl(B°, 0\bigr)}{G}.
    \]
\end{enumerate}

\subsubsection*{The image of $β°$}

Consider the topological closure $Z := \overline{\img β°}$, which is an analytic
subset of $B$ because $β°$ is quasi-algebraic.  We write $Z° := Z ∩ B°$ and set
$p := \dim Z$.  The following has been shown in \cite[Obs.~8.12]{orbiAlb2}.
\begin{enumerate}
  \item The variety $Z°$ is positive-dimensional.

  \item\label{il:5-5-6} The variety $Z°$ is a proper subset $Z° ⊊ B°$.
\end{enumerate}

\subsubsection*{Differentials on $B°$}

Finally, \cite[Obs.~8.11]{orbiAlb2} employs methods from Kawamata's proof of the
Bloch conjecture, in order to find $B°$-invariant differentials $τ°_0$, …, $τ°_p
∈ H⁰ \bigl( B°,\, Ω^p_{B°} \bigr)$ with the following properties.
\begin{enumerate}
  \item The restrictions $τ°_•|_{Z°_{\reg}}$ are linearly independent
  top-differentials on $Z°_{\reg}$, and therefore define a $(p+1)$-dimensional
  linear system $L ⊆ H⁰ \bigl( Z°_{\reg}, ω_{Z°_{\reg}} \bigr)$.
  
  \item The associated meromorphic map $φ_L : Z°_{\reg} \dasharrow ℙ^p$ is
  generically finite.
\end{enumerate}

\subsection*{Step 3: Setup}
\approvals{Erwan & yes \\ Stefan & yes}

Let $V$ be the normalized fibre product $ℂ ⨯_{X°} \what{X}°$, which may be
reducible or irreducible.  The construction of $V$ extends
Diagram~\eqref{eq:5-1} as follows,
\[
  \begin{tikzcd}[column sep=2cm, row sep=1cm]
    V \ar[d, two heads, "γ_V\text{, quotient}"'] \ar[r, "\what{φ}\text{, dense img.}"'] \ar[rrr, bend left=10, "\what{f} := β°◦\alb°◦\what{φ}"] & \what{X}° \ar[r, "\alb°"'] \ar[d, two heads, "γ_{\what{X}°}\text{, quotient}"] & \Alb° \ar[r, two heads, "β°"'] \ar[d, two heads, "γ_{\Alb°}\text{, quotient}"] & B° \ar[d, two heads, "γ_{B°}\text{, quotient}"] \\
    ℂ \ar[r, "φ\text{, dense img.}"] \ar[rrr, bend right=10, "f := ε°◦δ°◦φ"'] & X° \ar[r, "δ°"] & \factor{\Alb°}{G} \ar[r, two heads, "ε°"] & \factor{B°}{G}.
  \end{tikzcd}
\]
We highlight two elementary facts that will later become relevant.

\begin{claim}\label{claim:5-6}%
  The morphism $f$ is a $\cC$-morphism from $(ℂ, 0)$ to the quotient pair
  $\factor{\bigl(B°, 0\bigr)}{G}$.
\end{claim}
\begin{proof}[Proof of Claim~\ref{claim:5-6}]
  This follows from \cite[Prop.~11.1]{orbiAlb1}, given that the quotient pair
  $\bigl(B°, 0\bigr)/G$ is uniformizable.
  \qedhere~\mbox{(Claim~\ref{claim:5-6})}
\end{proof}

\begin{claim}\label{claim:5-7} %
  The natural $G$-action on $V$ is effective.  More precisely: if $g ∈ G$ is any
  element, then the fixed point set of the associated translation $V → V$ is
  finite.
\end{claim}
\begin{proof}[Proof of Claim~\ref{claim:5-7}]
  If $g ∈ G$ is any hypothetical element whose translation morphism fixes an
  entire component $V' ⊂ V$, then equivariance of $\what{φ}$ implies that the
  image set $\what{φ}(V')$ is $g$-fixed.  But that image set is dense in
  $\what{X}°$.  \qedhere~(Claim~\ref{claim:5-7})
\end{proof}

\subsection*{Step 4: Resolution of singularities}
\approvals{Erwan & yes \\ Stefan & yes}

Consider the semitoric variety $B° ⊂ B$ and choose a log-resolution $π : Y → Z$
of $(Z, Δ_B ∩ Z)$.  Write $Δ_Y$ for the reduced snc divisor on $Y$ whose support
equals $π^{-1}(Δ_B)$.  As before, write $Y° := Y ∖ Δ_Y$.  The following diagram
summarizes the setup,
\[
  \begin{tikzcd}[column sep=2cm, row sep=1cm]
    & Y \ar[d, "π\text{, resolution}"]\\
    V \ar[r, "\what{f}"'] \ar[d, two heads, "γ_V\text{, quotient}"'] \ar[ur, bend left=5, "φ"]& Z \ar[d, two heads, "\text{quotient}"] \ar[r, hook, "ι\text{, inclusion}"'] & B \ar[d, two heads, "\text{quotient}"] \\
    ℂ \ar[r, "f"'] & \factor{Z}{G} \ar[r, hook] & \factor{B}{G}.
  \end{tikzcd}
\]

\begin{rem}\label{rem:5-8}%
  If $V' ⊂ V$ is any component, then it is clear by construction that
  $\what{f}(V')$ is Zariski-dense in $Z$.  The morphism $φ$ is the canonical
  lifting of $\what{f}$ to the resolution of singularities.  This lifting exists
  because the images $\what{f}(V')$ are Zariski-dense and hence not contained in
  the indeterminacy locus of $π^{-1}$.
\end{rem}

\begin{claim}\label{claim:5-9} %
  The log pair $(Y,Δ_Y)$ is of log-general type.
\end{claim}
\begin{proof}[Proof of Claim~\ref{claim:5-9}]
  Recall from \cite[Prop.~3.15]{orbiAlb2} that the $B°$-invariant differentials
  $τ°_• ∈ H⁰ \bigl( B°,\, Ω^p_{B°} \bigr)$ extend to logarithmic differentials
  $τ_• ∈ H⁰ \bigl(B,\, Ω^p_B(\log Δ) \bigr)$.  Pulling those back, we obtain
  sections in $ω_Y(\log Δ_Y)$ such that the meromorphic map of the associated
  linear subsystem of $|K_Y+Δ_Y|$ is generically finite.
  \qedhere~(Claim~\ref{claim:5-9})
\end{proof}

\begin{rem}
  Claim~\ref{claim:5-9} implies that the manifold $Y$ is Moishezon.  In
  particular, there exists a blow-up $\wtilde{Y} → Y$ where $\wtilde{Y}$ is
  projective, \cite[Cor.~6.10]{MR1326624}.  Replacing $Y$ by its blow-up, we may
  assume without loss of generality that the manifold $Y$ is projective.
\end{rem}

\begin{claim}\label{claim:5-11} %
  The Albanese morphism $\alb(Y,Δ_Y)°$ of the log pair $(Y,Δ_Y)$ is generically
  injective.  The dimension of the Albanese satisfies $\dim \Alb(Y,Δ_Y)° > \dim
  Y$.
\end{claim}
\begin{proof}[Proof of Claim~\ref{claim:5-11}]
  Given that $Y°$ admits a generically injective, quasi-algebraic morphism into
  $B°$, generic injectivity of $\alb(Y,Δ_Y)°$ follows directly from the
  universal property, as spelled out in \cite[Def.~4.2]{orbiAlb2}.  For the
  inequality between the dimension, recall from \ref{il:5-5-6} that $Z°$ is a
  proper subset of $B°$.  But \cite[Proposition~4.10]{orbiAlb2} implies that
  $Z°$ generates $B°$ as a group, so that the natural morphism $\Alb(Y,Δ_Y)° →
  B°$ is necessarily surjective.  \qedhere~(Claim~\ref{claim:5-11})
\end{proof}

\begin{claim}\label{claim:5-12} %
  The $G$-action on $B°$ is not free.  In particular, there exists a
  non-trivial, cyclic subgroup $H ⊂ G$ that acts on $B°$ with a fixed point.
\end{claim}
\begin{proof}[Proof of Claim~\ref{claim:5-12}]
  Claim~\ref{claim:5-11} allows applying the Logarithmic Bloch-Ochiai Theorem
  \cite[Thm.~4.8.17]{MR3156076} to the manifold $Y$ and the divisor $Δ_Y$:
  entire curves $ℂ → Y°$ cannot have Zariski dense images.  Together with
  Remark~\ref{rem:5-8} this implies in particular that no component of $V$ is
  isomorphic to $ℂ$.  The quotient morphism $γ_V$ must therefore be branched,
  and there do exist group elements $g ∈ G$ that fix certain points of $V$.
  Equivariance of $\what{f}$ will then imply that $g$ fixes their images in
  $B°$.  \qedhere~(Claim~\ref{claim:5-12})
\end{proof}

\subsection*{Step 5: Cyclic subgroups of $G$ and differentials on $B$}
\approvals{Erwan & yes\\ Stefan & yes}

In the situation at hand, where $Z°$ is not contained in the translate of any
quasi-algebraic subgroup of $B°$, the results of Section~\ref{sec:2} can be
interpreted as an existence statement for differentials with certain factors of
automorphy.

\begin{claim}\label{claim:5-13} %
  If $H ⊆ G$ is cyclic and if its action on $B°$ has a fixed point, then there
  exists a logarithmic differential $τ_H ∈ H⁰ \bigl( B,\, Ω¹_B(\log Δ_B) \bigr)$
  such that the following holds.
  \begin{enumerate}
    \item\label{il:5-13-1} The pull-back differential $σ_H := (\diff \what{f})τ_H$
    does not vanish identically on any component of $V$.

    \item\label{il:5-13-2} If $h ∈ H ∖ \{e_H\}$ is any element with associated
    translation $t_h : V → V$, then there exists a number $ζ ∈ ℂ^* ∖ \{ 1\}$
    such that $(\diff t_h)σ_H = ζ·σ_H$.
  \end{enumerate}
\end{claim}
\begin{proof}[Proof of Claim~\ref{claim:5-13}]
  We use the notation introduced in Setting~\vref{setnot:2-1}.  By assumption,
  the variety $Z$ is not contained in any proper sub-semitorus of $\Alb°$, and
  then neither are the sets $\what{f}(V')$, where $V' ⊂ V$ is any component.
  According to Lemma~\vref{lem:2-4}, this implies that none of the restricted
  morphisms $\what{f}|_{V'}$ has its image tangent to the foliation $ℰ^*_{H,0}$.
  It follows that there exists a number $λ > 0$ and a form $τ ∈ E_{H,λ}$ such
  that $(\diff \what{f}|_{V'})τ ≠ 0$, for every component $V' ⊂ V$.  But then,
  Remark~\ref{rem:2-2} immediately implies that there exists a number $λ > 0$
  and a form $τ_H ∈ E_{H,λ}$ such that \ref{il:5-13-1} holds.
  Property~\ref{il:5-13-2} is now an immediate consequence of the description of
  the $H$-action on differentials, as given in \eqref{eq:2-1-2}.
  \qedhere~(Claim~\ref{claim:5-13})
\end{proof}

\begin{notcho}
  Let $Γ ⊂ \sP(G)$ be the set of non-trivial, cyclic subgroups of $G$ whose
  action on $B°$ has at least one fixed point; Claim~\ref{claim:5-12} guarantees
  that this set is not empty.  For each of the finitely many $H ∈ Γ$, choose one
  differential form $τ_H ∈ H⁰ \bigl( B,\, Ω¹_B(\log Δ_B) \bigr)$ that satisfies
  the conclusion of Claim~\ref{claim:5-13} and write
  \begin{align*}
    ω_H & := π^* τ_H && ∈ H⁰\bigl( Y,\, Ω¹_Y(\log Δ_Y) \bigr) \\
    σ_H & := g^* ω_H && ∈ H⁰\bigl( V,\, Ω¹_V \bigr).  \\
    \intertext{Following Notation~\ref{not:3-9}, we denote the associated meromorphic
    functions of $V$ as}
    ξ_H & := η(σ_H) && ∈ H⁰ \bigl( V,\, \sK_V \bigr).
  \end{align*}
  Maintain this choice for the remainder of the present proof.
\end{notcho}

\subsection*{Step 6: End of proof}
\approvals{Erwan & yes \\ Stefan & yes}

In order to derive a contradiction and to finish the proof of
Proposition~\ref{prop:5-2}, we show that the degeneracy criterion of
Theorem~\vref{thm:4-1} applies to the morphism $φ$ and to the finite collection
of differentials, $\{ω_H \::\: H ∈ Γ \}$.  Claims~\ref{claim:5-9} and
\ref{claim:5-11} together with the following two assertions ensure that the
assumptions of Theorem~\ref{thm:4-1} are indeed satisfied.

\begin{claim}\label{claim:5-15}%
  For every subgroup $H ∈ Γ$, the meromorphic function $ξ_H$ is holomorphic.
\end{claim}
\begin{proof}[Proof of Claim~\ref{claim:5-15}]
  Let $H ∈ Γ$ be any group.  To see that $ξ_H$ is holomorphic, recall from
  Claim~\ref{claim:5-6} that $f$ is a $\cC$-morphism between $(ℂ, 0)$ and $(B°,
  0)/G$.  It will then follow directly from the definition of a
  ``$\cC$-morphism'' in \cite[Def.~8.1]{orbiAlb1} that the differential form
  $σ_H ∈ H⁰ \bigl( V,\, Ω¹_V \bigr)$ is a section of the sheaf $Ω¹_{(ℂ, 0, ρ)} =
  ρ^* Ω¹_ℂ$.  \qedhere~(Claim~\ref{claim:5-15})
\end{proof}

\begin{claim}\label{claim:5-16}%
  For every point $v ∈ \Ramification ρ$, there exists one subgroup $H ∈ Γ$ such
  that $ξ_H$ vanishes at $v$.
\end{claim}
\begin{proof}[Proof of Claim~\ref{claim:5-16}]
  Given any point $v ∈ \Ramification γ_V$, observe that its isotropy group $H$
  is non-trivial.  Claim~\ref{claim:5-7} and the classic statement about
  ``linearization at a fixed point'', \cite[Sect.~1.5]{MR782881}, implies that
  the natural representation morphism
  \[
    G_v → \operatorname{Gl}( T_V|_v ) ≅ \operatorname{Gl}( 1,\, ℂ ) ≅ ℂ^*
  \]
  is injective.  In particular, $H$ is isomorphic to a subgroup of $ℂ^*$ and
  hence cyclic.  The fact that $\what{f}$ is equivariant implies that
  $\what{f}(v)$ is an $H$-fixed point of $B°$.  In summary, we find that $H ∈
  Γ$.  Choose a generator $h ∈ H$ and recall that there exists a number $ζ ∈ ℂ^*
  ∖ \{ 1\}$ such that $(\diff t_h)σ_H = ζ·σ_H$.  Since $ρ^* dt$ is
  $G$-invariant, this implies
  \[
    ξ_H◦ t_h = ζ·ξ_H.
  \]
  In particular, the function $ξ_H$ must necessarily vanish at the $H$-fixed
  point $v ∈ V$.  The claim thus follows.  \qedhere~(Claim~\ref{claim:5-16})
\end{proof}

Theorem~\ref{thm:4-1} now asserts that the morphism $φ$ is algebraically
degenerate, and then so are $f_W$ and $f$.  This contradicts our assumption and
ends the proof of Proposition~\ref{prop:5-2}.  \qed


\end{document}